\documentclass[12pt]{amsart} 
\usepackage[utf8x]{inputenc}
\usepackage[T1]{fontenc}
\usepackage{tikz}
\usepackage{amssymb}
\usepackage{amsmath}
\usepackage{hyperref}
\usepackage{mystyle}
\usepackage{enumerate}
\usepackage{StandardNewCommands}
\usepackage{mathtools}
\usepackage[foot]{amsaddr}

\usetikzlibrary{arrows}

\title[Restrictions on endomorphism algebras over a given field]{Endomorphism algebras of  abelian varieties with large cyclic 2-torsion field over a given field}

\author{Pip Goodman}
\address{Universitat de Barcelona\\
ORCID: 0000-0001-6735-2367}
\email{pip.goodman@ub.edu}
\date{}

\begin{document}

\begin{abstract}
    In this article we study the endomorphism algebras of abelian varieties $A$ defined over a given number field $K$ with large cyclic  2-torsion fields.
    A key step in doing so is to provide criteria for all the endomorphisms of $A$ to be defined over $K(A[2])$, the field extension generated by its 2-torsion.

    When $K= \QQ$ and $\Gal(\QQ(A[2])/\QQ)$ is cyclic of prime order $p = 2 \dim(A) +1$, we prove that there are only finitely many possibilities for the geometric endomorphism algebra $\End(A) \otimes \QQ$.
    In fact, when $\dim (A) \not \in \{3,5,9,21,33,81\}$, we show $\End(A) \otimes \QQ$ is a proper subfield of the $p$-th cyclotomic field.
    In particular, when $g=2$, $\End(A) \otimes \QQ$ is isomorphic to either $\QQ$ or $\QQ(\sqrt{5})$.
\end{abstract}

\maketitle

\section{Introduction}
Faltings' Isogeny Theorem \cite[p. 360, Satz 4]{Faltings} tells us that for an abelian variety $A$, defined over a number field $K$, we can recover its ring of endomorphisms defined over $K$, $\End_K(A)$, tensored by $\ZZ_\ell$  (where $\ell$ is a prime) from looking at the action of the absolute Galois group $G_K \coloneqq \Gal(\overline{K}/K)$ on its $\ell$-adic Tate module $T_\ell(A)$.
More precisely, Faltings' Isogeny Theorem tells us that the natural injection
\[\End_K(A) \otimes \ZZ_\ell \hookrightarrow \End_{K}(T_\ell(A)) \]
is an isomorphism.

The mod $\ell$ version of the above is known to fail to be an isomorphism in general, that is, generally speaking, the injection
\[\End_K(A) \otimes \ZZ/\ell \ZZ \hookrightarrow  \End_{K}(A[\ell])\]
is not an isomorphism.
However, when the action of $G_K$ on $A[\ell]$ is particularly `large' one might hope that $\End_{K}(A[\ell]) \; (= \End_{\overline{K}}(A[\ell])^{G_K})$ is small enough to place significant restrictions on $\End_K(A) \otimes \ZZ/\ell \ZZ$ and hence $\End_K(A)$.

Zarhin has had much success with carrying this out for jacobians of hyperelliptic curves (with $\ell =2$), see for example \cite{Zar00,Zar_modreps,Zarhin_clifford}.
The attraction of this case is that for a hyperelliptic curve determined by a smooth affine model $y^2=f(x)$, the Galois group $\Gal(K(A[2])/K)$ is equal to $\Gal(K(f)/K)$, the Galois group of the splitting field of $f$, and its action on $A[2]$ can be described explicitly via its action on the roots of $f$.
Thus giving explicit examples.

In the cases considered by Zarhin, all of the Galois groups are insoluble and act 2-transitively on the roots of $f$.
In \cite{Goodman_Restrictions} the author considered restrictions on $\End_K(A)$ and $\End(A)\coloneqq\End_{\overline{K}}(A)$ when $\Gal(K(A[\ell])/K)$ contained merely an element of large prime order, with no restrictions on $K$.

However, considering the case of elliptic curves, for example, where the possible endomorphism algebras depend solely on $[K \colon  \QQ]$, it is natural to impose conditions on $K$.
The main advantage in doing so is that this allows for more arithmetic input, whereas the proofs of the above results are heavily representation theoretic.

Thus, in this paper we study restrictions on $\End(A)$ (and $\End_K(A)$) coming from the ground field in conjunction with the results of \cite{Goodman_Restrictions}.
In this way, this paper may be seen as a natural continuation of \cite{Goodman_Restrictions}.

As an example, we obtain (where $\zeta_n$ denotes a primitive $n$-th root of unity):

\begin{theorem}[$ \subset$ Corollary \ref{thm_Cp_field} + Theorem \ref{thm_Cp_over_imag_quad}]
\label{intro_thm_Cp_field}
Let $A/K$ be an abelian variety of dimension $g \geq 2$ with $p=2g+1$ prime.
Suppose $\Gal(K(A[2])/K)$ has order $p$ and either $K=\QQ$ or $K$ is an imaginary quadratic field such that $p \nmid \# \mathrm{Cl}(K)$.
Then either
\begin{itemize}
    \item $\End(A) \otimes \QQ$ is a proper subfield of $\QQ(\zeta_p)$; or
    \item $p \in \{7,11,19,43,67,163\}$ and $A$ is isogenous over $\overline{K}$ to the power of an elliptic curve with complex multiplication by $\QQ(\sqrt{-p})$.
\end{itemize}
In particular there are only finitely many possibilities for $ \End(A)\otimes \QQ$.
\end{theorem}

What is particularly interesting here is that this does not fit into the philosophy of the above results, which tell us to expect that the bigger the Galois group $\Gal(K(A[2])/K)$ is, the smaller the endomorphism algebra of $A$ should be.
Indeed, the jacobian defined by $y^2=x^5-2$ has CM by $\QQ(\zeta_5)$ and the Galois group of the splitting field of $x^5-2$ has order 20.
Whereas by the above, an abelian surface $A$ over $\QQ$ with $\Gal(\QQ(A[2])/\QQ)$ of order 5 must have endomorphism algebra a proper subfield of $\QQ(\zeta_5)$.
This is of course only possibly because we are dealing the 2-torsion opposed to the entire 2-adic Tate module.

Intuitively, the reason for this discrepancy is already exemplified in the above example: $\Gal(K(A[2])/K)$ being small should force the endomorphisms to be defined over a small extension of $K$.
To make this rigorous, we need to study the minimal extension $L/K$ over which $A$ obtains all its endomorphisms.
This field was christened in \cite{GK17} as the \emph{endomorphism field} of $A$ .

An important result in understanding this field was obtained by Silverberg, who showed it is contained in $K(A[n])$ for $n \geq 3$ \cite[Thm. 2.4]{Silverberg}.
Bounds on $[L\colon K]$ were then studied in \cite{GK17} and also in \cite{FKRS12,SatoTateThreefolds} via its connection to Sato-Tate groups.

However, as eluded to above, we wish to be able to bound $L$ relative to $K(A[2])$.
The following theorem allows us to do so (see also Theorems \ref{thm:endofield3mod4}, \ref{thm:endofield1mod4} for an analogue in the setting of QM surfaces):

\begin{theorem}[$=$ Theorem \ref{thm:endo_field_contained_in_2_torsion_field_when_endos_are_totally_inert_at_2}]
\label{thm:intro:endo_field_in_2_torsion}
Suppose $E \cong \End(A)\otimes \QQ$ is a Galois extension of $\QQ$ and $L \nsubseteq K(A[2])$.
Then $[E\colon \QQ]$ is even.

If moreover $\End(A)$ is 2-maximal in $E$, then $\Gal(E/\QQ)$ has a non-trivial normal elementary abelian 2-subgroup which is contained in the intersection of inertia groups $\cap _{\p |2} I_\p$ where $\p$ runs over all primes above $2$ in $E$.

In particular, if $E/\QQ$ is Galois, $\End(A)$ is a 2-maximal order and $2$ is not wildly ramified in $E$, then $L \subseteq K(A[2])$.
\end{theorem}

Recall that an order $\OO$ in a number field $E$ is said to be \emph{2-maximal} if any of the equivalent conditions hold: its index in the maximal order $\OO_E$ is odd; $\OO \otimes \ZZ_2 = \OO_E\otimes \ZZ_2$; or $\OO \otimes \ZZ_{(2)} = \OO_E \otimes \ZZ_{(2)}$ \cite[Lem. 9.5.3, Lem. 9.6.7]{Voight_quaternions_book}.
Example \ref{example_2maximal_necessary} shows that the condition $\End(A)$ is 2-maximal cannot be removed.

Finally, let us underline the following corollary of Theorems \ref{intro_thm_Cp_field},\ref{thm:intro:endo_field_in_2_torsion}, where we write $\End^0(A) \coloneqq \End(A) \otimes \QQ$.

\begin{corollary}[$=$ Corollary \ref{cor:genus_two_C5_over_Q}]
\label{cor:intro:genus_two_C5_over_Q}
Let $A/\QQ$ be an abelian surface. Suppose $\Gal(\QQ(A[2])/\QQ) \cong C_5$.
Then either $\End(A)=\ZZ$ or $\End_\QQ^0(A)=\End^0(A)=\QQ(\sqrt{5})$.
\end{corollary}

Both cases are possible, see Example \ref{example_both_C5_cases_possible}.


\begin{acknowledgements}

The author thanks Tim Dokchitser, Francesc Fité, Xavier Guitart, Jef Laga, Jeremy Rickard, V\'{i}ctor Rotger, Rowan Swiers, Ciaran Schembri, John Voight, David Zywina and the anonymous referee for useful comments, questions and discussions.    

The author thanks the University of Bristol and the Université Clermont Auvergne for their hospitality during the period this paper was written.
Where the author was funded by the EPSRC (grant number EP/N509619/1) and an early career fellowship from the LMS respectively.
Finally, the author thanks Universitat de Barcelona where final revisions to this paper were made whilst he was supported by the Spanish Ministry of Science and Innovation via the grant “Abelian varieties, L-functions, and rational points” (code PID2022-137605NB-I00).
\end{acknowledgements}

\begin{notation}
Throughout $A$ will denote an abelian variety of dimension $g$ defined over a number field $K$.
Moreover, fixing an algebraic closure of $\overline{K}$ of $K$, we denote by $L$ the smallest extension of $K$ contained in $\overline{K}$ over which all the endomorphisms of $A$ are defined.

For an extension $F/K$ contained in $\overline{K}$, we write $\End_F(A)$ for the ring of endomorphisms of $A$ defined over $F$, $\End(A)  \coloneqq \End_{\overline{K}}(A)$ and $\End_F^0(A)\coloneqq \End_F(A) \otimes \QQ$.
Letting $G_F$ denote the absolute Galois group of $F$, we have $\End(A)^{G_F}=\End_F(A)$.
We write $\mathrm{Cl}(F)$ for the class group of $F$ and $\#\mathrm{Cl}(F)$ for its order.

We denote by $C_n, D_n, F_n$ the cyclic group of order $n$, the dihedral group of order $2n$, and the Frobenius group isomorphic to $\FF_n \rtimes \FF_n^*$ (for $n$ a prime power) respectively.

A primitive $n$-th root of unity is denoted by $\zeta_n$. For a polynomial $f\in K[x]$, we denote its splitting field by $K(f)$, and write $\Gal(f)$ for $\Gal(K(f)/K)$.
\end{notation}

\section{Proof of the main theorems}

Silverberg has shown the endomorphism field $L$ of an abelian variety $A/K$ is contained in the $n$-torsion field for every $n \geq 3$ \cite[Thm. 2.4]{Silverberg}.
The case $n=4$ implies $\Gal(L/L \cap K(A[2]))$ is an elementary abelian 2-group, since $\Gal(K(A[4])/K(A[2]))$ is too, for details see \cite[Prop. 3.9]{Goodman_Restrictions}.
Below we provide conditions on $\End(A)$ which guarantee $L \subseteq K(A[2])$.

\begin{theorem}
\label{thm:endo_field_contained_in_2_torsion_field_when_endos_are_totally_inert_at_2}
Suppose $E \cong \End^0(A)$ is a (finite) Galois extension of $\QQ$ and $L \nsubseteq K(A[2])$.
Then $[E\colon \QQ]$ is even.

If moreover $\End(A)$ is 2-maximal in $E$, then $\Gal(E/\QQ)$ has a non-trivial normal elementary abelian 2-subgroup which is contained in the intersection of inertia groups $\cap _{\p |2} I_\p$ where $\p$ runs over all primes above $2$ in $E$.

In particular, if $E/\QQ$ is Galois, $\End(A)$ is a 2-maximal order and $2$ is not wildly ramified in $E$, then $L \subseteq K(A[2])$.
\end{theorem}

\begin{proof}
Suppose $\Gal(L/K)$ acts faithfully on $\End(A) \otimes \ZZ/2\ZZ$.
Recall that the natural map $\End(A) \rightarrow \End(A[2])$ induces an injection $\End(A) \otimes \ZZ/2\ZZ \hookrightarrow \End(A[2])$.
The group $\Gal(\overline{K}/K(A[2]))$ acts trivially on $\End(A[2])$ and hence trivially on $\End(A) \otimes \ZZ/2\ZZ$ as well.
Since $\Gal(L/K)$ acts faithfully on $\End(A) \otimes \ZZ/2\ZZ$, we find $L \subseteq K(A[2])$.
As we are assuming $L \not \subseteq K(A[2])$, we deduce that the action of $\Gal(L/K)$ on $\End(A) \otimes \ZZ/2\ZZ$ has a non-trivial element $\gamma$ in its kernel.

The action of $\Gal(L/K)$ on $\End(A)$ is faithful, allowing us to view $\Gal(L/K)$ as a subgroup of $\Aut_{\ZZ}(\End(A)) \cong \GL_d(\ZZ)$ where $d$ is the rank of $\End(A)$ viewed as a free $\ZZ$-module.
The action also extends to field automorphisms of $\End^0(A) \cong E$, thus we may also view $\Gal(L/K)$ as a subgroup of $\Gal(E/\QQ)$.
Any finite subgroup in the kernel of the reduction map $\GL_d(\ZZ) \rightarrow \GL_d(\ZZ/2\ZZ)$ is an elementary abelian 2-group \cite[p. 151, Lem. 9]{Guralnick_Lorenz}.
As $\Aut_{\ZZ/2\ZZ}(\End(A) \otimes \ZZ/2\ZZ) \cong \GL_d(\ZZ/2\ZZ)$, we deduce that $\{\sigma \in \Gal(L/K)| (\sigma-1)\End(A) \subseteq 2 \cdot \End(A)\}$ is an elementary abelian $2$-group which is non-trivial since it contains $\gamma$.
Hence $[E\colon \QQ]$ is divisible by 2.

Assume now $\End(A)$ is a 2-maximal order in $E$.
Hence $\End(A) \otimes \ZZ_2 \cong \OO_E \otimes \ZZ_2$ where $\OO_E$ is the maximal order of $E$.
As $\Gal(E/\QQ)$ acts on $\OO_E$, it also acts on $\OO_E \otimes \ZZ_2$ and thus also on $\End(A) \otimes \ZZ_2$ via the above isomorphism.
The kernel $H$ of the composite map $\Gal(E/\QQ) \hookrightarrow\Aut_{\ZZ_2}(\End(A) \otimes \ZZ_2) \rightarrow \Aut_{\ZZ/2\ZZ}(\End(A) \otimes \ZZ/2\ZZ)$ is an elementary abelian 2 group \cite[p. 151, Lem. 9]{Guralnick_Lorenz}\footnote{The statement they give does not apply directly to $\ZZ_2$, however the proof carries over verbatim.}, which is non-trivial since it contains $\gamma$.

Write $2 \OO_E = \prod_{i} \p^{e_i}_i$ where each $\p_i$ is a prime ideal.
We have that every element of $H$ acts trivially on $\End(A) \otimes \ZZ/2\ZZ \cong \OO_E \otimes \ZZ/2\ZZ \cong \OO_E/2\OO_E \cong \prod_{i} \OO_E/\p^{e_i}_i$.
In particular, every element of $H$ acts trivially on each $\OO/\p^{e_i}_i$, so $H$ is contained in the inertia group $I_{\p_i}$ for every prime $\p_i$ above 2 in $E$.
\end{proof}

\begin{example}
\label{example_2maximal_necessary}
The condition that $\End(A)$ is 2-maximal cannot be removed. Indeed, the elliptic curve $y^2 =  (x+2)(x^2 - 2x - 11)$ has CM by $\ZZ[\sqrt{-3}]$ and its 2-torsion field is $\QQ(\sqrt{3})$, see \cite[Appendix A \S 3]{SilvermanII}.
\end{example}

\begin{proposition}
\label{lemma_ramification_plus_my_thm}
Let $A/K$ be an abelian variety of dimension $g \geq 1$ with $p=2g+1$ prime.
Suppose $\q$ is a prime of bad reduction for $A$ and there is an element of order $p$ in the image of the inertia group $I_\q$ in $\Gal(K(A[\ell'])/K)$ for some prime $\ell' \nmid pN(\q)$.
Then either
\begin{itemize}
    \item $p$ does not divide $[L\colon K]$ and $\End^0(A)$ is isomorphic to a subfield of $\QQ(\zeta_p)$; or
    \item $p$ divides $[L\colon K]$, $g\geq 3$ and $A$ is isogenous over $\overline{K}$ to the power of an absolutely simple abelian variety with complex multiplication by a proper subfield of $\QQ(\zeta_p)$.
\end{itemize}
In particular there are only finitely many possibilities for $\End^0(A)$.
\end{proposition}

\begin{proof}
 Suppose $p$ divides $[L\colon K]$. Then by \cite[Thm. 2.5]{Goodman_Restrictions}, $g\geq 3$ and $A$ is isogenous over $\overline{K}$ to the power of an absolutely simple abelian variety  with complex multiplication by a proper subfield $F$ of $\QQ(\zeta_p)$.
Hence we may suppose $p$ does not divide $[L\colon K]$.

By assumption, there exists a prime $\ell' $ not dividing $p N(\q)$ such that the image of $I_\q$ in $\Gal(K(A[\ell'])/K)$ contains an element of order $p$, say $\tau_{\ell'}^{(1)}$.
For any $n \in \ZZ_{\geq 1}$, the kernel of the restriction map $r_n \colon \Gal(K(A[(\ell')^{n+1}])/K) \rightarrow \Gal(K(A[(\ell')^n])/K)$ has order a power of $\ell'$.
Thus as $\ell' \neq p$, we have that for all $n\in \ZZ_{\geq 1}$, there is an element $\tau_{\ell'}^{(n+1)} \in \Gal(K(A[(\ell')^{n+1}])/K)$ of order $p$ contained in the image of $I_\q$ such that $r_n(\tau_{\ell'}^{(n+1)})= \tau_{\ell'}^{(n)} $.
In particular, there exists an element $\tau_{\ell'}$ in the image of $I_\q$ in $\Aut_{\ZZ_{\ell'}}(T_{\ell'}(A))$ of order $p$. 

Let $\ell$ be a prime and let $\rho_\ell \colon G_K \rightarrow \Aut_{\ZZ_\ell}(T_\ell(A))$ denote the representation of $G_K$ arising from its action on the Tate module $T_\ell(A)$.
By \cite[Thm. 4.3]{SGA7}, for any $\sigma \in I_\q$ and prime $\ell \nmid N(\q)$, the characteristic polynomial of $\rho_\ell(\sigma)$ has rational coefficients independent of $\ell$.

Thus the trace of $\tau_{\ell'}$ is rational.
It follows that its eigenvalues are the primitive $p$-th roots of unity, owing to the fact $\tau_{\ell'}$ has at most $2g = p-1$ distinct eigenvalues.
Hence the characteristic polynomial of $\tau_{\ell'}$ is the $p$-th cyclotomic polynomial.
Let $\tau \in I_\q$ be such that $\rho_{\ell'}(\tau)=\tau_{\ell'}$.
We deduce that for any prime $\ell \nmid N(\q)$, the characteristic polynomial of $\tau_\ell \coloneqq\rho_\ell(\tau)$ is the $p$-th cyclotomic polynomial.
 
 By Dirichlet's theorem on arithmetic progressions, we can find a prime $\ell''\nmid N(\q)$, which is a primitive root modulo $p$.
The reduction of $\tau_{\ell''}$ modulo $\ell''$ lands in $\Gal(K(A[\ell''])/K)$ and has order $p$ thanks to the form of its characteristic polynomial.
 This allows us to apply \cite[Thm. 2.9]{Goodman_Restrictions} and deduce $E \coloneqq \End^0(A)$ is a field.
 
 As $p$ does not divide $[L\colon K]$, for any prime $\ell \nmid N(\q)$, the element $\tau_\ell$ lies in the image of $G_L$ in $\Aut_{\ZZ_\ell}(T_\ell(A))$.
 The decomposition $E \otimes \QQ_\ell = \prod_{\lambda|\ell} E_\lambda$ induces a decomposition $V_\ell(A)=\prod_{\lambda|\ell}V_\lambda(A)$ giving representations $G_L \rightarrow \GL_n(E_\lambda)$ where $n=\frac{2g}{[E\colon\QQ]}$, see \cite[Ch. II]{ribet_RM} for further details.
 Let $\tau_\lambda \in \GL_n(E_\lambda)$ be the projection of $\tau_\ell$ onto $V_\lambda(A)$.
By the above, the eigenvalues of $\tau_\lambda$ are distinct primitive $p$-th roots of unity.
Taking the trace of $\tau_\lambda$ we deduce $E_\lambda$ contains a subfield of $\QQ(\zeta_p)$ of degree $[E\colon\QQ]$.

As $\QQ(\zeta_p)$ has a unique subfield $F$ of degree $[E\colon\QQ]$ and for all, but possibly finitely many primes $\lambda$, $E_\lambda$ contains $F$, the Chebotarev Density Theorem implies $E$ contains $F$.
But $[E\colon \QQ] = [F\colon \QQ]$, so in other words, $E \subseteq \QQ(\zeta_p)$.
\end{proof}

By imposing a condition on the ray class groups of primes above $2$ in $K$, we deduce an explicit version of the above:

\begin{theorem}
\label{thm_explicit_general_finitude}
Let $A/K$ be an abelian variety of dimension $g \geq 1$ with $p=2g+1$ prime.
Suppose $\Gal(K(A[2])/K) \cong C_p$ and $p$ divides neither the class number of $K$, nor the multiplicative order of the residue field of any prime above 2. Then either
\begin{itemize}
    \item $p$ does not divide $[L\colon K]$ and $\End^0(A)$ is isomorphic to a subfield of $\QQ(\zeta_p)$; or
    \item $p$ divides $[L\colon K]$, $g \geq 3$ and $A$ is isogenous over $\overline{K}$ to the power of an absolutely simple abelian variety with complex multiplication by a proper subfield of $\QQ(\zeta_p)$.
\end{itemize}
In particular there are only finitely many possibilities for $\End^0(A)$.
\end{theorem}

\begin{proof}
Let $F/K$ be an odd degree abelian extension unramified outside of 2.
Then by class field theory, any prime dividing $[F\colon K]$ divides either the class number of $K$, or the multiplicative order of the residue field of a prime above 2 in $K$.

It follows that $K(A[2])/K$ is ramified at some prime $\q \nmid 2$. This allows us to apply the above proposition and conclude.
\end{proof}

The following lemma is well-known.
\begin{lemma}
\label{Lemma_reflex_field}
Let $A/K$ be an absolutely simple abelian variety with CM by a Galois extension $E/\QQ$. Then $L=E^*K$ and $E \supseteq E^*$, where $E^*$ is the reflex field of $E$. Furthermore if $E/\QQ$ is abelian, then $E=E^*$. 
\end{lemma}

\begin{proof}
As $E/\QQ$ is Galois, the reflex field $E^*$ is a subfield of $E$ \cite[Prop. 28, p. 62]{Shimura_CM_book}.
Moreover, since  $A$ is absolutely simple its CM type is primitive.
Applying \cite[Prop. 30, p. 65]{Shimura_CM_book} we find the endomorphism field $L$ equals $E^*K$.
Finally, if $E/\QQ$ is an abelian extension, then $E=E^*$ by  \cite[Example (1), p. 63]{Shimura_CM_book}.
\end{proof}

\begin{corollary}
\label{thm_Cp_field}
Let $A/\QQ$ be an abelian variety of dimension $g \geq 1$ with $p=2g+1$ prime.
Suppose $\Gal(\QQ(A[2])/\QQ) \cong C_p$. Then either
\begin{itemize}
    \item $\End^0(A)$ is a proper subfield of $\QQ(\zeta_p)$; or
    \item $p \in \{7,11,19,43,67,163\}$ and $A$ is isogenous over $\overline{\QQ}$ to the power of an elliptic curve with complex multiplication by $\QQ(\sqrt{-p})$.
\end{itemize}

In particular there are only finitely many possibilities for $\End^0(A)$.
\end{corollary}

\begin{proof}
Clearly, we may apply Theorem \ref{thm_explicit_general_finitude}.
Suppose first $p$ divides $[L\colon\QQ]$.
Then $g \geq 3$ and $A$ is isogenous over $\overline{\QQ}$ to the power of an absolutely simple abelian variety with complex multiplication by a proper (non-trivial) subfield $F$ of $\QQ(\zeta_p)$.
Thus there is some minimal finite extension $K/\QQ$ such that $A_{K}$ contains a simple abelian subvariety $B$ with CM by $F$.
The above lemma applies to $B$ and shows $L$, the endomorphism field of $A$, contains $KF$.

By \cite[Prop. 3.9]{Goodman_Restrictions} the group $\Gal(L/\QQ)$ is an extension of $C_p$ by $C_2^m$ for some $m$.
In particular, any non-trivial element of even order in $\Gal(L/\QQ)$ has either order 2 or $2p$.

Since $p$ does not divide $[F\colon\QQ]$, the cyclic quotient $\Gal(F/\QQ)$ of $\Gal(L/\QQ)$ has order two.
Using $F$ is a CM field only ramified at $p$, we find $F=\QQ(\sqrt{-p})$ where $p\equiv 3 \mod{4}$.
In particular $g$ is odd.

We shall now show $F$ has class number one, which by the Baker--Heegner--Stark Theorem will conclude the proof in this case.
As $F/\QQ$ is a quadratic extension, $B$ is an elliptic curve.
Thus by CM theory \cite[Thm. II.4.1, p. 121]{SilvermanII} $KF$ contains $H$ the Hilbert class field of $F$.
In particular, $L$ contains $H$, so $[H\colon F]$ divides $2^mp$ for some $m$.
On the other hand, \cite[Cor. 2.16]{FiteGuitart_power_of_ellipticcurves_SatoTate_realisability} tells us every element of $\Gal(H/F)$ has order dividing $g$.
As $g$ is odd and less than $p$, we find $\Gal(H/F)=1$ as claimed.

We now suppose $p$ does not divide $[L\colon \QQ]$ and show $\End^0(A) \not \cong \QQ(\zeta_p)$.
Suppose we had equality, then applying Lemma \ref{Lemma_reflex_field} we find $L = \QQ(\zeta_p)$.
First let us note that as they are extensions of $\QQ$ of coprime degree, we have $L \cap \QQ(A[2])= \QQ$.
If $p=3$, then \cite[Thm. 2.9]{Goodman_Restrictions} implies $\End(A)$ is a 2-maximal order in $\End^0(A) \cong \QQ(\zeta_3)$, so Theorem \ref{thm:endo_field_contained_in_2_torsion_field_when_endos_are_totally_inert_at_2} gives us that $ L \subseteq \QQ(A[2])$.
But $L \cap \QQ(A[2])=\QQ$ so $L= \QQ$, contradicting the above.
Thus we may take $p \geq 5$.
Since $L \cap \QQ(A[2])= \QQ$, \cite[Prop. 3.9]{Goodman_Restrictions} implies $\Gal(L/\QQ)$ is an elementary abelian 2-group.
However, $\Gal(\QQ(\zeta_p)/\QQ)$ is not an elementary abelian $2$-group, thus we once again have a contradiction.
\end{proof}

Let us record some easy corollaries of \cite[Thm. 2.9]{Goodman_Restrictions}, Theorems
\ref{thm:endo_field_contained_in_2_torsion_field_when_endos_are_totally_inert_at_2}, \ref{thm_Cp_field} and Lemma \ref{Lemma_reflex_field}.

\begin{corollary}\label{cor:C3_elliptic_curve_has_no_CM}
Let $K$ be a number field with a real embedding.
Let $f \in K[x]$ be an irreducible cubic polynomial, whose discriminant is a square in $K$.
Then the elliptic curve defined by $y^2=f(x)$ does not have complex multiplication. 
\end{corollary}

\begin{proof}
    Let $A$ denote the elliptic curve defined by $y^2=f(x)$.
    The assumptions on $f$ imply $\Gal(f) \cong C_3$.
    Let us suppose for a contradiction that $A$ has complex multiplication. 
    As $2$ is a primitive root modulo 3, we may apply \cite[Thm. 2.9]{Goodman_Restrictions}.
    Since $\QQ(\zeta_3)$ does not have any proper CM subfields, the second case of \cite[Thm. 2.9]{Goodman_Restrictions} cannot hold.
    It follows that $(2)$ is inert in $E \coloneqq\End^0(A)$ and $\End(A)$ is 2-maximal in $E$.

    By Lemma \ref{Lemma_reflex_field}, the endomorphism field of $A$ is $KE$.
    As $E$ does not have any real embeddings, it is not a subfield of $K$.
    Thus $[KE\colon K]=2$.
    Since $K(A[2])=K(f)$ is an odd degree extension of $K$, it does not contain $KE$.
    Thus by Theorem \ref{thm:endo_field_contained_in_2_torsion_field_when_endos_are_totally_inert_at_2} and the above, we find that $(2)$ is ramified in $E/\QQ$.
    But this contradicts the above, so we find $A$ cannot have complex multiplication as claimed.
\end{proof}

\begin{corollary}
\label{cor:genus_two_C5_over_Q}
Let $A/\QQ$ be an abelian surface. Suppose $\Gal(\QQ(A[2])/\QQ) \cong C_5$.
Then either $\End(A)\cong \ZZ$ or $\End_\QQ^0(A)=\End^0(A)\cong \QQ(\sqrt{5})$.
\end{corollary}

\begin{proof}
    By Corollary \ref{thm_Cp_field}, we have either $\End^0(A) \cong \QQ$ or $\QQ(\sqrt{5})$.
    In the first case $\End(A) \cong \ZZ$.
    Thus we may suppose $\End^0(A)\cong\QQ(\sqrt{5})$.
    
    We are left to show that $L$, the endomorphism field of $A$, equals $\QQ$.
    The injection $\Gal(L/\QQ) \hookrightarrow \Aut_\QQ(\End^0(A)) \cong \Gal(\QQ(\sqrt{5})/\QQ)$ implies the degree $[L\colon\QQ]$ divides 2.
    Since $\QQ(A[2])=\QQ(f)$ is an odd degree extension of $\QQ$, we find $L \cap \QQ(A[2]) = \QQ$.
    Thus $L$ is contained in $\QQ(A[2])$ if and only if $L=\QQ$.

    As 2 is a primitive root modulo 5, we may apply \cite[Thm. 2.9]{Goodman_Restrictions} which implies that $\End(A)$ is 2-maximal in $\End^0(A) \cong \QQ(\sqrt{5})$.
    Since 2 is inert in $\QQ(\sqrt{5})$, Theorem \ref{thm:endo_field_contained_in_2_torsion_field_when_endos_are_totally_inert_at_2} implies that $L \subseteq \QQ(A[2])$.
    Thus $L=\QQ$ as claimed.
\end{proof}

\begin{example}
\label{example_both_C5_cases_possible}
    Both cases are possible.
Indeed, examples of the first case are easily found.
For the second case, note the jacobian $J$ of the hyperelliptic curve $y^2 = x(x^5 - 4x^4 + 2x^3 + 5x^2 - 2x - 1)$ has $\End_\QQ(J) = \End(J) \cong \ZZ\left[\frac{1+\sqrt{5}}{2}\right]$ by \cite[Prop. 1]{Wilson_paper} and the Galois group of $x^5 - 4x^4 + 2x^3 + 5x^2 - 2x - 1$ has order 5.
\end{example}

We present a variant of Theorem \ref{thm_explicit_general_finitude} for abelian varieties over imaginary quadratic fields.

\begin{theorem}
\label{thm_Cp_over_imag_quad}
Let $g \geq 2$ be an integer and suppose $p=2g+1$ is prime.
Let $K$ be an imaginary quadratic field of class number coprime to $p$.
Let $A/K$ be an abelian variety of dimension $g$.
Suppose $\Gal(K(A[2])/K) \cong C_p$. Then either
\begin{itemize}
    \item $\End^0(A)$ is isomorphic to a proper subfield of $\QQ(\zeta_p)$; or
    \item $p \in \{7,11,19,43,67,163\}$ and $A$ is isogenous over $\overline{K}$ to the power of an elliptic curve with complex multiplication by $\QQ(\sqrt{-p})$.
\end{itemize}
In particular there are only finitely many possibilities for $\End^0(A)$.
\end{theorem}

\begin{proof}

Clearly, we may apply Theorem \ref{thm_explicit_general_finitude}.
Suppose first $p$ divides $[L\colon K]$.
Then $A$ is isogenous over $\overline{K}$ to the power of an absolutely simple abelian variety $B$ with complex multiplication by a proper subfield $F$ of $\QQ(\zeta_p)$.
Lemma \ref{Lemma_reflex_field} applies to $B$ and shows $L$, the endomorphism field of $A$, contains $F$.

By \cite[Prop. 3.9]{Goodman_Restrictions} the group $\Gal(L/K)$ is an extension of $C_p$ by $C_2^m$ for some $m$.
In particular, any non-trivial element of even order in $\Gal(L/K)$ has either order 2 or $2p$.
Since $p$ does not divide $[F\colon\QQ]$, the cyclic quotient $\Gal(FK/K)$ of $\Gal(L/K)$ has order dividing two.

In turn we deduce $[F\colon\QQ]$ divides $4$ and is equal to 4 only if $F \supsetneq K$.
Being a CM field, $[F\colon\QQ]=4$ would imply $F$ contains a real quadratic field in addition to the imaginary quadratic field $K$.
Thus $\Gal(F/\QQ) \cong C_2 \times C_2$ contradicting the fact $F \subseteq \QQ(\zeta_p)$.
Thus $[F\colon\QQ]=2$.
As $F$ is a CM field only ramified at $p$, we deduce $F=\QQ(\sqrt{-p})$ where $p\equiv 3 \mod{4}$.

Let $H$ be the Hilbert class field of $F=\QQ(\sqrt{-p})$.
By \cite[Thm. 2.14]{FiteGuitart_power_of_ellipticcurves_SatoTate_realisability} the group $\Gal(HK/FK)$ has order dividing $g$, thus arguing as in the proof of Theorem \ref{thm_Cp_field}, we deduce $[HK \colon FK]=1$.
In other words, $H \subseteq FK = K(\sqrt{-p})$.

Thus either $H=\QQ(\sqrt{-p})$ or $H = K(\sqrt{-p})$.
Suppose the latter holds.
Then, as $H$ is an unramified extension of $\QQ(\sqrt{-p})$, which itself is only ramified at $p$, we see that $p$ is the only finite prime ramified in $K$.
As $p$ is odd, $K$ is tamely ramified at $p$, so by the Kronecker-Weber Theorem $K \subseteq \QQ(\zeta_p)$.
This implies $K = \QQ(\sqrt{-p})$, since it is the unique quadratic field contained in $\QQ(\zeta_p)$.
Thus $H = K(\sqrt{-p}) = \QQ(\sqrt{-p})$.

In particular, $\QQ(\sqrt{-p})$ has class number one.
As $p \equiv 3 \mod{4}$ and $g \geq 2$, it follows, from the classification of imaginary quadratic fields of class number one, that $p \in \{7,11,19,43,67,163\}$.

We now suppose $p$ does not divide $[L \colon K]$ and show $\End^0(A) \not \cong \QQ(\zeta_p)$. Suppose we had equality, then by Lemma \ref{Lemma_reflex_field}, $L = K(\zeta_p)$.
 Applying \cite[Prop. 3.9]{Goodman_Restrictions} shows $\Gal(L/K) = \Gal(K(\zeta_p)/K) \cong \Gal(\QQ(\zeta_p)/ \QQ(\zeta_p) \cap K )$ has order at most $2$.
 Hence $[\QQ(\zeta_p) \colon  \QQ]$ divides $4$, being equal to 4 only if $K$ is contained in $\QQ(\zeta_p)$.
 But $[\QQ(\zeta_p) \colon  \QQ] = 4$ implies $p=5$ and $\QQ(\zeta_5)$ does not have an imaginary subfield. Whence $[\QQ(\zeta_p) \colon  \QQ]= 2$ and $p=3$.
 This in turn implies $g=1$ which we have ruled out by assumption.
\end{proof}

\begin{remark}
The condition on the class number of $K$ cannot be removed.
Indeed, the polynomial $f(x) = x^5-19x^4+107x^3+95x^2+88x-16$ has Galois group $D_5$ and its splitting field is the Hilbert class field of $K \coloneqq \QQ(\sqrt{-131})$.
The jacobian $J_f$ of the hyperelliptic curve defined by $y^2=f(x)$, has endomorphism algebra isomorphic to $\QQ(\sqrt{13})$ and
the group $\Gal(K(J_f[2])/K)$ is isomorphic to $C_5$.
\end{remark}

The above example raises the following question:

\begin{question}
\label{question}
    Let $K$ be an imaginary quadratic field with class number divisible by 5.
    Let $H$ be a degree 5 extension of $K$ contained in the Hilbert class field of $K$.
    What are the possible endomorphism algebras for abelian surfaces $A/\QQ$ with $\QQ(A[2]) = H$?
\end{question}

\begin{remark}
    With reference to the above question, there are already several restrictions on the possible endomorphism algebras.
    Indeed, by combining Theorem \ref{thm:endo_field_contained_in_2_torsion_field_when_endos_are_totally_inert_at_2}, \cite[Thm. 2.10]{Goodman_Restrictions} and \cite[Prop. 30, p. 65; Example (2), p. 64]{Shimura_CM_book}, we see that either $\End^0(A) \cong \QQ$ or $\QQ(\sqrt{d})$ with $d \equiv 5 \mod{8}$.
    The above question then becomes: given $H$, which $d$ occur?
\end{remark}

We finish this section with the following extension of \cite[Thm. 3.5]{Goodman_Restrictions}.

\begin{theorem}
\label{Field_def_Frob}
 Let $f \in K[x]$ be a polynomial of odd degree $n$ with Galois group isomorphic to a Frobenius group $G$ of order $n(n-1)$.
 Let $J_f$ be the jacobian associated to the hyperelliptic curve defined by $y^2=f(x)$.

Suppose ${\rm End}^0(J_f)$ is isomorphic to a number field $E$ of dimension $s$ over $\QQ$.
Then $E/\QQ$ is Galois with ${\rm Gal}(E/\QQ) $ isomorphic to a quotient of $H$ the Frobenius complement of $G$. 

Furthermore, $L/K$ is an extension of degree $s$ contained in $K(f)$, and as abstract groups, ${\rm Gal} (L/K) \cong  {\rm Gal}(E/\QQ)$.
Moreover, if $s=n-1$, then $L=EK$.

Finally, if $\End(J_f)$ is $2$-maximal, then $E$ is unramified at 2.
\end{theorem}

\begin{proof}
This result is  \cite[Thm. 3.5]{Goodman_Restrictions} with the extra assertion that if $s=n-1$, then $L=EK$.
If $s=n-1$, then $J_f$ is an absolutely simple abelian variety with CM by $E$. We may therefore apply Lemma \ref{Lemma_reflex_field} to find $L=E^*K$ and $E^* \subseteq E$. Hence $[E^* \colon \QQ] \geq [L \colon K] = s = [E \colon \QQ]$. It follows we have equality $E^*=E$.
\end{proof}

\begin{example}
For $f(x) = x^5-2$, it is well known the endomorphism algebra of $J_f$ is $\QQ(\zeta_5)$ and $\QQ(\zeta_5)$ is the unique degree 4 extension contained in the splitting field $\QQ(f)$.

A more interesting example is given by the genus 2 curve with LMFDB label \href{https://www.lmfdb.org/Genus2Curve/Q/28561/a/371293/1}{28561.a.371293.1} first found in \cite{LMFDB_genus2_curves} (though we use an odd degree model computed using \texttt{Magma} \cite{magma}).
Here $f(x) =  52x^5 + 104x^4 + 104x^3 + 52x^2 + 12x + 1$, $\Gal(f) \cong F_5$ and $J_f$ has CM by the number field $E$ defined by $x^4 - x^3 + 2x^2 + 4x + 3$.
The unique degree 4 extension contained in  $\QQ(f)$ is given by $E$.

This field is totally inert at 2 (as predicted by \cite[Thm. 2.9]{Goodman_Restrictions}) and unramified outside 13.
We note its class number is one, but in line with the theorems presented above, $\QQ(f)/K$ is a degree 5 extension ramified only at 2.
\end{example}

\section{An analogue of Theorem \ref{thm:intro:endo_field_in_2_torsion} for QM surfaces}
\label{Section_Quaternion_Algebras}
Let $B = \QQ + i\QQ + j\QQ + k\QQ$ be an indefinite quaternion algebra generated by $i,j$ satisfying $i^2 = D/m$, $j^2 = m$, $ij=-ji$ and $k=ij$, where $D$ is a positive squarefree integer and $m|D$.
We note that as $B$ is indefinite, $0> -m = \mathrm{n}(j)$, the reduced norm of $j$.
As we shall only use the reduced norm, reduced trace and reduced discriminant, we shall refer to these simply as the norm, trace and discriminant respectively.

An order of $B$ is said to be hereditary if it has squarefree discriminant (for alternate, equivalent, definitions see  \cite{DR, Voight_quaternions_book}).
We shall take $\OO$ to be a hereditary order in $B$ of discriminant $D$.

 An abelian surface $A/K$ is said to have quaternion multiplication by $\OO$ if there is an isomorphism $\iota \colon \OO \xrightarrow{\sim} \End(A)$.
We shall assume there is an element $\mu \in \OO$ satisfying $\mu^2+D =0$.
This is always the case when $\OO$ is a maximal order \cite[43.6.6, p. 818]{Voight_quaternions_book} and furthermore such an element $\mu$ induces a principal polarisation on $A$ \cite[p. 821]{Voight_quaternions_book}.
Fix such an $A$, $\mu$ and $\iota$.

Following \cite[Definition 3.3]{DR}, we say $\chi \in B$ is a \emph{twist} of $(\OO,\mu)$ if it lies in both $\OO$ and the normaliser $N_{B^*}(\OO)$, has trace zero, $\mathrm{n}(\chi)$ divides $D$, and $\mu\chi = - \chi\mu$.
The existence of such an element can be verified by a finite computation, we refer the reader to \cite[Rem., p. 9]{DR} for further details.

Owing in part to the fact $(\OO,\mu)$ has been fixed, there are only two possible values\footnote{This follows from \cite[Lemmas 3.5 and 3.7]{Rotger_field_moduli}.
Indeed, in the notation of the paper, we have $F=\QQ$ and $D \neq 3$, as $D$ is divisible by an even number of primes.
This forces $\omega_{odd}=1$.
The cited lemmas in turn show $C_2 \cong U_0 \leq V_0 \cong C_2 \times C_2$.
The non-trivial element of $U_0$ can be represented by $\mu = \sqrt{-D}$.
This allows us to write $V_0 = \langle [\mu], [\chi] \rangle$ where the representative $\chi$ may be taken to have reduced norm $m | D$.

Moreover, any representative of $[\chi]$ has reduced norm $m$ up to a rational square.
Likewise, representatives of the class $[\mu \chi]$ have reduced norm $D/m$ up to rational squares.
} $\mathrm{n}(\chi)$ can take and the product of these values is equal to $D$. 
We call $(\OO,\mu, \chi)$ a \emph{twisted principally polarised order}, and say $(\OO, \mu, \chi)$ is of discriminant $D$ and norm $-\mathrm{n}(\chi)$.

As before, we let $L$ denote the endomorphism field of $A$.
Dieulefait and Rotger \cite[Thm. 3.4]{DR} showed $\Gal(L/K)$ is isomorphic to one of the trivial group, $C_2$, or $C_2 \times C_2$.
Moreover, they proved in each case $\End_K^0(A)$ is respectively isomorphic to $B$; one of $\QQ(\mu), \QQ(\chi), \QQ(\mu \chi) $; or $\QQ$.
This determines (and heavily restricts) the possible images of
\[\Gal(L/K) \hookrightarrow \Aut_\ZZ(\End(A)) \cong N_{B^*}(\OO)/\QQ^*\]
arising from the natural action of $\Gal(L/K)$ on $\End(A)$.

To gain information on the intersection $L \cap K(A[2])$, we will use a description of orders of discriminant $D$ in $B$.

\begin{lemma}
\label{lem:maximalorder3mod4}
Suppose $D$ is even and $m \equiv 3 \mod{4}$. Then \[\OO = \ZZ + \frac{1}{2}(1 + j + k)\ZZ + \frac{1}{2}(1 + j - k)\ZZ + \frac{1}{2}(i+k)\ZZ\] is an order of discriminant $D$ in $B$.
Moreover, any order $\OO'$ of discriminant $D$ which contains $\ZZ[1,mi,j,k]$, satisfies $\OO' \otimes \ZZ_{(2)} = \OO \otimes \ZZ_{(2)}$.
\end{lemma}

\begin{proof}
It is a routine calculation to show $\OO$ is an order of discriminant $D$ (for an example, see \cite[p. 85 - 86]{vigneras}).




As $m$ is odd we have $ \ZZ[1,i,j,k] \otimes \ZZ_{(2)}=\ZZ[1,mi,j,k] \otimes \ZZ_{(2)} \subseteq \OO'\otimes \ZZ_{(2)}$.
There exists a lattice $\OO''$ such that $\OO''\otimes \ZZ_{(2)} =\OO' \otimes \ZZ_{(2)}$ and $\OO'' \otimes \ZZ_{(p)} = \ZZ[1,i,j,k] \otimes \ZZ_{(p)}$ for every odd prime $p$ \cite[Thm. 9.4.9]{Voight_quaternions_book}.
Moreover, $\OO''$ is in fact an order since every localisation is \cite[Lem. 10.2.10]{Voight_quaternions_book}.
As $\ZZ[1,i,j,k] \otimes \ZZ_{(p)}$ is contained in $\OO'' \otimes \ZZ_{(p)}$ for every prime $p$,  we deduce $\ZZ[1,i,j,k]$ is contained in $\OO''$ \cite[Cor. 9.4.7]{Voight_quaternions_book}.

The discriminant of $\ZZ[1,i,j,k]$ equals $4D$.
As the discriminant of an order is determined locally \cite[15.2.13]{Voight_quaternions_book}, we deduce the discriminant of $\OO''$ equals $D$.
Applying \cite[Lem. 15.2.5]{Voight_quaternions_book}, we see that any element contained in $\OO''$ but not in $\ZZ[1,i,j,k]$ is of the form $\frac{1}{4}\alpha$ with $\alpha \in \ZZ[1,i,j,k]$.

Let us write such an element as $ \frac{1}{4} \alpha = w +xi+yj+zk$.
Considering the trace of $\frac{1}{4} \alpha$, which is an integer \cite[Cor. 10.3.3]{Voight_quaternions_book}, we find $w \equiv 0 \mod{2}$.
Likewise as $\mathrm{n}(\frac{1}{4} \alpha) \in \ZZ$ \cite[Cor. 10.3.3]{Voight_quaternions_book}, we have $\mathrm{n}(\alpha) \equiv 0 \mod{16}$.
Considering $\mathrm{n}(\alpha) \equiv 0 \mod{2}$, we find $y \equiv 0 \mod{2}$, which in turn combined with $\mathrm{n}(\alpha) \equiv 0 \mod{4}$ implies $x^2+z^2 \equiv 0 \mod{4}$ and thus $x , z \equiv 0 \mod{2}$.

Hence it suffices to check elements of the form $\frac{1}{2}(a+bi+cj+dk)$ with $a,b,c,d \in  \{0,1\}$.
We have $\mathrm{n}(a+bi+cj+dk)  \equiv a-2b+c+2d \equiv 0 \mod{4}$, from which we deduce $\frac{1}{2}(1 + j + k)$, $ \frac{1}{2}(1 + i + j)$, $\frac{1}{2}(i+k)$ are the only integral such.
In order for the discriminant of $\OO''$ to equal $D$ we see that all of these elements must belong to $\OO''$.
Whence $\OO \otimes \ZZ_{(2)} = \OO'' \otimes \ZZ_{(2)} = \OO'\otimes \ZZ_{(2)}$.
\end{proof}
 
\begin{lemma}
\label{lem:maximalorder1mod4}
Suppose $m  \equiv 1 \mod{4}$. Then \[\OO = \ZZ + \frac{1}{2}(1 + j)\ZZ + k\ZZ + \frac{1}{2}(i+k)\ZZ\] is an order of discriminant $D$ in $B$. Furthermore, if $D \equiv 1 \mod{4}$ then
\[\OO_1 = \ZZ + \frac{1}{2}(1 + i)\ZZ + j\ZZ + \frac{1}{2}(j+k)\ZZ\] is an order of discriminant $D$ in $B$.
Likewise, if $D \equiv 3 \mod{4}$ then
\[\OO_3 = \ZZ + \frac{1}{2}(1 + k)\ZZ + j\ZZ + \frac{1}{2}(i+j)\ZZ\] is an order of discriminant $D$ in $B$.

Moreover, any order $\OO'$ of even discriminant $D$ which contains $\ZZ[1,mi,j,k]$ satisfies $\OO' \otimes \ZZ_{(2)} = \OO \otimes \ZZ_{(2)}$.
Any order $\OO'$ of discriminant $D \equiv t \mod{4}$ with $t \in \{1,3\}$, which contains $\ZZ[1,mi,j,k]$, satisfies either $\OO' \otimes \ZZ_{(2)} = \OO \otimes \ZZ_{(2)}$ or $\OO' \otimes \ZZ_{(2)} = \OO_t \otimes \ZZ_{(2)}$.
\end{lemma}

\begin{proof}
The proof follows the same strategy as for Lemma \ref{lem:maximalorder3mod4}.
\end{proof}

\begin{theorem}
\label{thm:endofield3mod4}
Let $A/K$ be an abelian surface with QM by a twisted principally polarised order $(\OO, \mu, \chi)$ of discriminant $D$ and norm $m$, where $ \OO \subseteq B$.

If $D$ is even and $m  \equiv 3 \mod{4}$, then $L \subseteq K(A[2])$.
\end{theorem}

\begin{proof}
If $L=K$, then there is nothing left to show.
Thus let us suppose $L \neq K$.
As $\End(A)$ is a hereditary order, the results of \cite[Thm. 3.4]{DR} apply to show $\Gal(L/K) \leq C_2^2$ and a non-trivial element of $\Gal(L/K)$ acts on $\End(A)$ (possibly after scaling) by conjugation as one of $\mu,\chi$ or $\mu\chi$.

We look to determine the action of these elements on $\End(A) \otimes \ZZ/2\ZZ$.
To do so we shall consider a $\ZZ_{(2)}$-basis of $\End(A) \otimes\ZZ_{(2)}$ reduced modulo $2$.
By considering the algebraic relations they satisfy, we may assume $\mu = k$, $\chi =j $ and $\mu\chi = mi$.
Hence Lemma \ref{lem:maximalorder3mod4} allows us to take $\End(A) \otimes \ZZ_{(2)} = \ZZ_{(2)} + \frac{1}{2}(1 + j + k)\ZZ_{(2)} + \frac{1}{2}(1 + j - k)\ZZ_{(2)} + \frac{1}{2}(i+k)\ZZ_{(2)}$.

Let $X = \frac{1}{2}(1 + j + k)$, $Y = \frac{1}{2}(1 + j - k)$ and $Z = \frac{1}{2}(i+k)$.
Let us examine the action of $i,j$ and $k$ on the $\ZZ_{(2)}$-basis of $\End(A) \otimes \ZZ_{(2)}$ given by $1,X,Y,Z$.
Each of $i,j$ and $k$ fix $1$. For $i$ we have $iXi^{-1}=1-X$, $iYi^{-1}= 1 - Y$ and $iZi^{-1} = Z+Y-X$. For $j$ we have $jXj^{-1}=Y$, $jYj^{-1}=X$, $jZj^{-1}=-Z$.
Looking at the coefficients, we see the action remains faithful on $\End(A)\otimes \ZZ/2\ZZ$.

Recall that the natural map $\End(A) \rightarrow \End(A[2])$ induces an injection $\End(A) \otimes \ZZ/2\ZZ \hookrightarrow \End(A[2])$.
The group $\Gal(\overline{K}/K(A[2]))$ acts trivially on $\End(A[2])$ and hence trivially on $\End(A) \otimes \ZZ/2\ZZ$ as well.
We deduce that $\Gal(\overline{K}/K(A[2])) \subseteq \Gal(\overline{K}/L)$, which implies $L \subseteq K(A[2])$.
\end{proof}

\begin{example}
The hereditary assumption is necessary. The following example shows not only $L$ need not be contained in $\QQ(A[2])$ for a non-hereditary order, but also the result of Dieulefait and Rotger fails.

Let $J$ be the jacobian defined by the hyperelliptic curve associated to $y^2 +y=6x^5 +9x^4 -x^3 -3x^2$ with LMFDB label \href{https://www.lmfdb.org/Genus2Curve/Q/20736/l/373248/1}{20736.l.373248.1}.
This surface has QM by an order of (reduced) discriminant $6^2$ in $\left(\frac{2,3}{\QQ}\right)$. In particular, the order is not hereditary.
The endomorphism field has defining polynomial $x^8 + 4x^6 + 10x^4 + 24x^2 + 36$ and the two torsion field of $J$ is $\QQ(\sqrt{2}, \sqrt{3})$.

This curve, made easily available on the LMFDB \cite{lmfdb}, was first found in \cite{LMFDB_genus2_curves}. Computations linked to its endomorphism algebra were carried out in \texttt{Magma} using code from \cite{Costa_Mascot_Sijsling_Voight_compute_endos}.
\end{example}

For ease of notation, let $F = K(A[2])$.

\begin{theorem}
\label{thm:endofield1mod4}
Let $A/K$ be an abelian surface with QM by a twisted principally polarised order $(\OO, \mu, \chi)$ of discriminant $D$ and norm $m$, where $ \OO \subseteq B$.

Suppose $m \equiv 1 \mod{4}$.
Then

\begin{itemize}
    \item for $D$ even, $\End^0_{F}(A)$ contains $\QQ(\sqrt{m})$;
    \item for $D \equiv 1 \mod{4}$, $\End^0_{F}(A)$ contains at least one of $\QQ(\sqrt{m})$ and $\QQ(\sqrt{D/m})$;
    \item for $D \equiv 3 \mod{4}$, $\End^0_{F}(A)$ contains at least one of $\QQ(\sqrt{m})$ and $\QQ(\sqrt{-D})$.;
\end{itemize}
\end{theorem}

\begin{proof}
If $L=K$, then there is nothing left to show.
Thus let us suppose $L \neq K$.
As $\End(A)$ is a hereditary order, the results of \cite[Thm. 3.4]{DR} apply to show $\Gal(L/K) \leq C_2^2$ and a non-trivial element of $\Gal(L/K)$ acts on $\End(A)$ (possibly after scaling) by conjugation as one of $\mu,\chi$ or $\mu\chi$.

We look to determine the action of these elements on $\End(A) \otimes \ZZ/2\ZZ$.
To do so we shall consider a $\ZZ_{(2)}$-basis of $\End(A) \otimes\ZZ_{(2)}$ reduced modulo $2$.
By considering the algebraic relations they satisfy, we may assume $\mu = k$, $\chi =j $ and $\mu\chi = mi$.
Lemma \ref{lem:maximalorder1mod4} allows us to take $\End(A) \otimes \ZZ_{(2)}$ equal to one of the three given orders tensored by $\ZZ_{(2)}$.
As these orders differ by permuting $i,j$ and $k$, we only give details for the case $\End(A) \otimes \ZZ_{(2)} = \ZZ_{(2)} + \frac{1}{2}(1 + j)\ZZ_{(2)} + k\ZZ_{(2)} + \frac{1}{2}(i+k)\ZZ_{(2)}$.

Let $X = \frac{1}{2}(1 + j)$, $Y = k$ and $Z = \frac{1}{2}(i+k)$.
Let us examine the action of $i,j$ and $k$ on the $\ZZ_{(2)}$-basis of $\End(A) \otimes \ZZ_{(2)}$ given by $1,X,Y,Z$.
Each of $i,j$ and $k$ fix $1$. For $i$ we have $iXi^{-1}=1-X$, $iYi^{-1}= - Y$ and $iZi^{-1} = Z-Y$. For $j$ we have $jXj^{-1}=X$, $jYj^{-1}=-Y$, $jZj^{-1}=-Z$. Looking at the coefficients, we see $j$ acts trivially on $\End(A)\otimes \ZZ/2\ZZ$ whereas $i$ and $k$ act by the same involution.

Let $L'$ be the field fixed by the kernel of the action of $G_K$ on $\End^0(A) \otimes \ZZ/2\ZZ$.
As $\End(A) \otimes \ZZ/2\ZZ \hookrightarrow \End(A[2])$ we have $L' \subseteq F$.
The above calculation shows $L/L'$ is an at most quadratic extension.
In particular, by \cite[Thm. 3.4]{DR} the ring of endomorphisms $\End_{L'}^0(A)$ contains at least one of $\QQ(\sqrt{m}), \QQ(\sqrt{-D}),\QQ(\sqrt{D/m})$.

Either all elements of $\Gal(\overline{K}/L')$ fix $\End^0(A)$ or there exists $\tau \in \Gal(\overline{K}/L')$ which acts on $\End^0(A)$ via conjugation by $j$.
In the first case, $\End^0_{L'}(A)=\End^0(A)$, so we may suppose the latter case holds.
Let $x \in \End^0_{L'}(A)$, which by an abuse of notation, we also view as being an element of $B$ via the isomorphism $\End^0(A) \cong B$.
As $\tau$ fixes $ \End^0_{L'}(A)$, we have $x  = \tau \cdot x = jxj^{-1}$.
Thus $x$ commutes with $j$.
As $\QQ(j)$ is a maximal commutative subalgebra of $B$, we deduce $x \in \QQ(j)$, so under the above isomorphism $ \End^0_{L'}(A) \subseteq \QQ(j)$.
We thus deduce  $ \End^0_{L'}(A) = \QQ(j) \cong \QQ(\sqrt{m})$.
\end{proof}

\begin{example}
The jacobian $J$ of the hyperelliptic curve $y^2=-2x^6-12x^5-21x^4-10x^3-3x^2+6x+1$ has QM by the maximal order of $B = \left(\frac{3,5}{\QQ}\right)$ which has discriminant 15.
The endomorphism field of $J$ is $L= \QQ(\sqrt{3}, \sqrt{-3})$, we have $L \cap \QQ(A[2]) = \QQ(\sqrt{-3})$ and $\End^0_{\QQ(\sqrt{-3})}(J) = \QQ(\sqrt{-15})$. These calculations were performed in \texttt{Magma} using code from \cite{Costa_Mascot_Sijsling_Voight_compute_endos}. The curve comes from \cite{Quaternionic_modular_threefold}.
\end{example}



\bibliographystyle{alpha}
\bibliography{bib}
\end{document}